\documentclass[11pt,letterpaper]{amsart}
\usepackage[utf8]{inputenc}

\usepackage{preamble}
\usepackage{natbib,amsthm,amssymb, url, hyperref}
\usepackage{graphicx}

\title{Nonnegative conorms, regular matroids, and the tropical Schottky problem}
\author{Yelena Mandelshtam}

\address{Department of Mathematics, University of Michigan, Ann Arbor, MI 48109, USA}
\email{yelenam@umich.edu}

\thanks{This research was partially supported by the National Science Foundation under Award DMS-2402069 and by the Bob Moses Fund of the Institute for Advanced Study.}



\begin{document}

\begin{abstract}
    We prove that a positive-definite quadratic form has nonnegative conorms if and only if it is matroidal, resolving a conjecture of Dutour Sikiri\'c and Kummer. We show that the positive conorm support determines a regular matroid, and the positive conorms are the coefficients of a unimodular integral decomposition of the form. As a result, we show that a positive-definite form lies in the tropical Schottky locus if and only if its conorms are nonnegative and its positive conorm support matroid is cographic. This answers a question of Chua,
Kummer, and Sturmfels and provides algorithms for the tropical Schottky decision and recovery problems in every genus. 
\end{abstract}
\maketitle

\section{Introduction}

The classical Schottky problem asks which principally polarized complex abelian varieties are Jacobians of algebraic curves. A period matrix
$\tau$ in the Siegel upper-half space $\mathfrak H_g$ determines a principally polarized abelian variety, and the classical Schottky locus consists of those period matrices arising from genus $g$ curves. In genus four, the closure of this locus is a hypersurface cut out by the Schottky--Igusa modular form, a degree-$16$ polynomial in theta constants \cite{Igusa,chua2019schottky}.

Its tropical counterpart asks which positive-definite real quadratic
forms arise as tropical Jacobians. Let $L$ be a rank-$g$ lattice and $Q$ a positive-definite quadratic form on
$L_{\RR}$. We say that $Q$ lies in the tropical Schottky locus
$\mathfrak{J}^{\mathrm{trop}}_g$ if there is a genus $g$ metric graph
$\Gamma$ and an isomorphism
\[
    L \cong H_1(\Gamma,\ZZ)
\]
under which $Q$ becomes the period form of $\Gamma$.

For a metric graph, the period form has a decomposition
\[
    Q_\Gamma=\sum_{f\in E(\Gamma)}\ell_fa_f^2,
\]
where the $a_f$ are its integral coedges. These coedges form a unimodular representation of $M^*(\Gamma)$. We recall the construction in \cref{sec:motivation}.

\begin{definition}
A finite family of integral covectors
\[
    (a_e)_{e\in E}\subseteq L^*
\]
is \emph{unimodular} if, after choosing an integral basis of $L$, every
maximal minor of the matrix whose columns are the $a_e$ is equal to
$0$ or $\pm1$.

A positive-definite quadratic form $Q$ on $L_{\RR}$ is
\emph{matroidal} if it admits a decomposition
\[
    Q=\sum_{e\in E}\lambda_e a_e^2,
    \qquad \lambda_e>0,
\]
for a unimodular integral configuration $(a_e)_{e\in E}$. The union
of the corresponding cones is called the \emph{matroidal locus}
\cite{meloviviani2012matroidal}.
\end{definition}

A matroid is \emph{regular} if it has a totally unimodular
representation, and it is \emph{cographic} if it is the dual
$M^*(G)$ of the graphic matroid of some graph $G$. The cographic matroidal forms are
exactly the forms arising from tropical Riemann matrices. The image of the tropical Torelli map is the subfan indexed by cographic matroids
\cite[Theorem~5.2.4]{brannetti2011torelli}.

The tropical Schottky problem amounts to recognizing the
cographic part of the matroidal locus directly from the quadratic
form. To formulate our result, set $V=L/2L.$ For each parity class $\alpha\in V$, its \emph{vonorm} is
\[
    \mu_Q(\alpha)
    =
    \min_{\substack{x\in L\\ [x]=\alpha}} Q[x].
\]
For every nonzero character $e\in V^*$, the corresponding normalized \emph{conorm} is
\[
    \lambda_e(Q)
    =
    -2^{1-g}
    \sum_{\alpha\in V}
    (-1)^{e(\alpha)}\mu_Q(\alpha).
\]

The vonorms and conorms have a direct interpretation in terms of the classical Schottky problem. In \cite{chua2019schottky}, Chua, Kummer, and Sturmfels use period matrices of the form
$\tau(t)=P(t)+itQ$, as $t\longrightarrow\infty$ to tropicalize theta constants. They show that the leading exponential order of the theta constant with characteristic $m=(m',m'')$ is determined by the tropical theta constant $\Theta_{m'}(Q)$, which in the present notation is
$\Theta_u(Q)
=
-\frac14\min_{\substack{x\in L\ [x]=u}}Q[x]
=
-\frac14\mu_Q(u)$. Thus the vonorms are, up to scaling and sign, the tropical theta constants. Their Fourier transforms $\vartheta_e(Q)$ in the notation of \cite{chua2019schottky} satisfy
$\lambda_e(Q)=2^{3-g}\vartheta_e(Q).$ In genus four, tropicalizing the Schottky--Igusa modular form yields exactly the conorm inequalities $\vartheta_e(Q)\geq0$ for every $0\neq e\in(L/2L)^*$ \cite[Theorem~4.9]{chua2019schottky}.

When all conorms are nonnegative, let $E(Q)=\{e\in V^*:\lambda_e(Q)>0\}.$ The elements of $E(Q)$ represent a binary matroid, which we denote by $M(Q)=M(E(Q)).$
Under the nonnegativity hypothesis of our main theorem, $E(Q)$ is the nonzero conorm support, so $M(Q)$ is the theta matroid introduced in \cite{chua2019schottky}. We review this formalism, together with our  conventions, in \cref{sec:conorms}.

Chua, Kummer, and Sturmfels proved that if $Q$ lies in the tropical Schottky locus, then $M(Q)$ is cographic and the positive conorms recover the corresponding edge lengths. They posed the converse as the question

 \begin{question}\cite[Question 4.11]{chua2019schottky}\label{ques:CKS}
     Let $Q$ be a positive-definite matrix such that the matroid $M(Q)$ is cographic with positive weights. Does this imply that $Q$ is in the tropical Schottky locus?
 \end{question}
 
In \cite{sikiric2022iso}, Dutour Sikiri\'c and Kummer proposed the stronger conjecture 
\begin{conj}\cite[Conjecture 2]{sikiric2022iso}\label{conj:DSK}
    A positive-definite matrix $Q$ lies in the matroidal locus if and only if $\la_e(Q) \geq 0$ for all $0 \neq e \in V^*$.
\end{conj}
They proved this conjecture for $g\leq5$
\cite[Theorems~6.4--6.5]{sikiric2022iso}. Together with the genus-four description above, this identifies the tropical Igusa locus with the matroidal locus, while the tropical Schottky locus is its cographic subfan.

Our main result proves \cref{conj:DSK} in every dimension. As noted in \cite{sikiric2022iso}, this also gives an affirmative answer to \cref{ques:CKS} in every genus and, since our proof is constructive, it yields explicit algorithms for tropical Schottky decision and recovery. In genus four, we also show that the remaining tropicalizations of the Schottky--Igusa presentations considered in \cite{chua2019schottky} impose no additional conditions: they detect the matroidal locus, but not its cographic sublocus.

The proof is motivated by the relationship between Voronoi-relevant vectors and simple cycles in the cographic case, conjectured by Caporaso–Viviani in \cite{caporaso2010torelli} and proved by Amini in \cite{amini2010lattice}. Inspired by this perspective, we isolate the structures in this picture that are shared by graphs and regular matroids and show that conorm positivity forces them to occur without any regular representation being given in advance. Once this is established, the remainder of the argument is self-contained and combines ideas from polyhedral geometry, tropical geometry, and matroid theory. 

In the upcoming subsection, we describe the graph model that motivates the construction. It is not strictly necessary to understand the proofs in the later sections, however it explains why the constructions are natural ones to consider. 

We briefly outline the proof: Fourier inversion expresses the vonorms in terms of the positive conorm support, which defines a binary matroid. This allows us to identify its cocircuits with parity classes that have a unique antipodal pair of shortest representatives, equivalently with the Voronoi-relevant vectors. We then show that the shortest lifts associated to a modular triple of cocircuits
satisfy a signed integral relation.

These relations propagate through basis exchanges and imply that the shortest lifts of the fundamental cocircuits of any matroid basis form a $\ZZ$-basis of $L$. From the cocircuit lifts avoiding each $e\in E$, we construct an integral covector $a_e$. The resulting covectors form a unimodular representation of the support matroid.
Finally, we prove
\[
    Q=\sum_{e\in E}\lambda_ea_e^2
\]
by constructing representatives in every parity class on which the two quadratic forms agree.

The rest of the paper is organized as follows. \cref{sec:conorms} reviews conorms and introduces the binary support matroid. \cref{sec:cocircuits} relates its cocircuits with Voronoi-relevant vectors. \cref{sec:modulartriples,sec:regularity} construct an integral unimodular representation and prove regularity. \cref{sec:reconstructing} reconstructs the quadratic form, and finally \cref{sec:main} contains the proofs of the main results and derives the tropical Schottky consequences and algorithms. The proof of \cref{conj:DSK} will follow as \cref{cor:nonnegative-conorm-criterion} from the stronger \cref{thm:main}.

\subsection{Graph flow lattices as motivation}\label{sec:motivation}

Let $G=(V(G),E(G))$ be a finite connected graph.  Choose an orientation of every edge.  The integral chain group is
\[
  C_1(G,\ZZ)=\ZZ^{E(G)},
\]
and the boundary map $\partial:C_1(G,\ZZ)\to \ZZ^{V(G)}$ sends an oriented edge to its head minus its tail.  The \emph{lattice of integer flows} is
\[
  \mathcal F(G)=\ker\partial=H_1(G,\ZZ).
\]
Its rank is the genus
\[
  g=|E(G)|-|V(G)|+1.
\]

If the edge $e$ has length $\ell_e>0$, then the metric graph defines the quadratic form
\[
  Q_G(z)=\sum_{e\in E(G)}\ell_e z_e^2,
  \qquad z=(z_e)_{e\in E(G)}\in \mathcal F(G).
\]
The restriction of the $e$th coordinate to $\mathcal F(G)$ is an integral covector
\[
  a_e:\mathcal F(G)\longrightarrow\ZZ,
  \qquad a_e(z)=z_e.
\]
Hence
\begin{equation*}
  Q_G=\sum_{e\in E(G)}\ell_e a_e^2.
\end{equation*}
The covectors $a_e$ are called \emph{coedges}.  A bridge has $a_e=0$ because no flow uses it. Our goal will be to recover such integral coordinate functionals $a_e$ from conorm positivity alone.

After choosing an integral basis of $H_1(G,\ZZ)$, place the coordinate columns of the $a_e$ into a matrix $A$.  Then the matrix of $Q_G$ is
\[
  Q_G=A\diag(\ell_e)A^{\mathsf T}.
\]
The column matroid of $A$ is the cographic matroid $M^*(G)$.

An oriented simple cycle gives a primitive flow with entries in $\{0,\pm1\}$.  It is support-minimal among nonzero flows.  In the cographic matroid $M^*(G)$, the cocircuits are the simple cycles of $G$.

For a graph flow lattice, these simple-cycle flows are exactly the Voronoi-relevant vectors.  More generally, Amini describes the full face poset of the Voronoi cell in terms of strongly connected orientations of subgraphs \cite{amini2010lattice}.  The part of that picture relevant here is that simple cycles correspond to cocircuits of the cographic matroid $M^*(G)$ and to the Voronoi-relevant vectors. 

The basic local relation among graph cycles is the \emph{theta relation}. As we prove in \cref{thm:metric-modular-triple}, conorm positivity forces the corresponding relation among shortest cocircuit lifts even when no graph or regular representation is known to exist. Suppose a subgraph consists of three internally disjoint paths $P_1,P_2,P_3$ joining the same two vertices.  It has three simple cycles
\[
  C_{12}=P_1-P_2,\qquad
  C_{13}=P_1-P_3,\qquad
  C_{23}=P_2-P_3,
\]
with the signed flow relation
\[
  C_{12}-C_{13}+C_{23}=0.
\]
In matroid language, the three cycle supports form a modular triple of cocircuits in the sense of \cref{def:modulartriple}. The heart of the argument in this paper is to show that the graph-theoretic theta relation persists without any graph or regular representation being given in advance: nonnegative conorms imply that the shortest lattice lifts of every modular triple of cocircuits satisfy the corresponding signed integral relation.

\vspace{1cm}
\textbf{Acknowledgements}\\
We thank Hadleigh Frost and Raul Penaguiao for conversations about Voronoi cells which have led to a much better understanding of the landscape. We also thank Mario Kummer for comments on the manuscript. Many ideas involved in this paper were discussed with ChatGPT. It was used as an exploratory tool for literature search, testing formulations, and checking intermediate arguments. The author takes full responsibility for all
statements and proofs in the paper. 

%This research was partially supported by the National Science Foundation under Award DMS2402069 and the Bob Moses Fund of the Institute for Advanced Study.

\section{Conorms and the binary support matroid}\label{sec:conorms}

In this section we develop the vonorm and conorm formalism, following \cite{CS}, summarized in the introduction and fix the notation used throughout the paper. We then give an intrinsic formulation of the theta matroid introduced by
Chua, Kummer, and Sturmfels \cite{chua2019schottky}. The remaining matroid-theoretic material is standard, but we record the short arguments and basis-exchange formulas needed later in order to fix notation.

Throughout, $L$ is a rank-$g$ lattice,
\[
    L_{\RR}=L\otimes_{\ZZ}\RR,
    \qquad
    L^*=\Hom(L,\ZZ),
\]
and $Q$ is a positive-definite quadratic form on $L_{\RR}$. Its associated symmetric bilinear form is
\[
    \langle x,y\rangle_Q
    =
    \frac{Q[x+y]-Q[x]-Q[y]}{2},
    \qquad
    Q[x]=\langle x,x\rangle_Q.
\]

Set $V=L/2L.$ This is a $g$-dimensional vector space over $\Ftwo$.  Its dual is $V^*=\Hom_{\Ftwo}(V,\Ftwo).$ For $e\in V^*$ and $\alpha\in V$, the expression $e(\alpha)$ is either $0$ or $1$.

\begin{definition}
Following Conway and Sloane \cite{CS}, for a parity class $\alpha\in V$, define the \emph{vonorm}
\[
  \mu(\alpha)=\min\{Q[x]:x\in L,\ x\bmod 2L=\alpha\}.
\]
Because $Q$ is positive-definite and every coset of $2L$ is discrete, the minimum exists.  Moreover, $\mu(\alpha) \geq 0$ for all $\alpha$, with equality if and only if $\alpha = 0$.
\end{definition}

For $e\in V^*$, let $\chi_e(\alpha)=(-1)^{e(\alpha)}.$ For a function  $f:V\to\RR$, use the unnormalized Fourier transform
\[
  \widehat f(e)=\sum_{\alpha\in V}\chi_e(\alpha)f(\alpha).
\]
Then Fourier inversion gives
\[
  f(\alpha)=2^{-g}\sum_{e\in V^*}\chi_e(\alpha)\widehat f(e).
\]

\begin{definition}
For $0\ne e\in V^*$, define the \emph{normalized conorm}
\begin{equation}\label{eq:lambda-def}
  \lambda_e=-2^{1-g}
  \sum_{\alpha\in V}(-1)^{e(\alpha)}\mu(\alpha)
  =-2^{1-g}\widehat\mu(e).
\end{equation}
\end{definition}

This classical normalization of Conway and Sloane \cite{CS} is convenient because, on a graph, $\lambda_e$ is an edge length after coedges with the same mod-$2$ class have been aggregated.  This matches the normalized conorm of Dutour Sikiri\'c--Kummer \cite{sikiric2022iso}. In the notation of Chua--Kummer--Sturmfels, after choosing an integral basis of $L$ and using the induced identification $V^*\cong(\mathbf F_2)^g$, one has $\lambda_e=2^{3-g}\vartheta_e(Q)$; the factor is positive and thus irrelevant to sign questions \cite{chua2019schottky}. 

Fourier inversion gives the following formula originally in \cite{CS}. We include a proof for completeness. 

\begin{proposition}\cite[Equation (9)]{CS}\label{prop:fourier-inversion}
For every $\alpha\in V$,
\[
  \mu(\alpha)=
  \sum_{\substack{0\ne e\in V^*\\e(\alpha)=1}}\lambda_e.
\]
\end{proposition}
\begin{proof}
    Fourier inversion at $0$ gives
\[
  \widehat\mu(0)+\sum_{e\ne0}\widehat\mu(e)=0.
\]
Observe
\begin{align*}
    \sum_{\substack{e\ne0\\e(\alpha)=1}}\lambda_e
  &=\frac12\sum_{e\ne0}\lambda_e\bigl(1-(-1)^{e(\alpha)}\bigr)\\
  &=2^{-g}\left(
  \widehat\mu(0)+
  \sum_{e\ne0}(-1)^{e(\alpha)}\widehat\mu(e)
  \right)\\
  &=\mu(\alpha).
\end{align*}
\end{proof}

From now on, we will assume, as in the hypothesis of \cref{conj:DSK}, that $\la_e \geq 0$ for all $0 \neq e \in V^*$. Let $$E = \{e \in V^*: \la_e > 0\}.$$

The following observation is already made by Conway and Sloane in \cite{CS} after their Equation 9.
\begin{lemma}[\cite{CS}]\label{lem:E-spans}
The set $E$ spans $V^*$.
\end{lemma}
\begin{proof}
    Suppose not, so $\Span E \subset V^*$. Then, since $V$ is the dual of $V^*$, there is some $0 \neq \alpha \in V$ such that $e(\alpha) = 0$ for all $e \in E$. However, then $\mu(\alpha) = 0$, a contradiction.
\end{proof}

Following \cite{chua2019schottky}, the nonzero conorm characteristics represent a binary matroid of rank $g$, which they call the theta matroid of $Q$. We give a basis-independent formulation using $V^*$.

The elements of $E$ are distinct nonzero vectors in $V^*$.  They represent a simple rank $g$ binary matroid $M=M(E)$. The matroid is defined as follows:
\begin{itemize}
\item a subset $S\subseteq E$ is independent if its elements are linearly independent in $V^*$;
\item $\rk_M(S)=\dim_{\Ftwo}\Span(S)$;
\item a basis is a $g$-element linearly independent subset.
\end{itemize}

Consider the evaluation map
\begin{equation*}
  \iota:V\longrightarrow\Ftwo^E,
  \qquad
  \iota(\alpha)=\bigl(e(\alpha)\bigr)_{e\in E}.
\end{equation*}
It is injective by \cref{lem:E-spans}.  Its image
\[
  C=\iota(V)\subseteq\Ftwo^E
\]
is the \emph{cocycle space} of $M$. The elements of $C$ are called \emph{cocycles}. For $\alpha\in V$, we write
\begin{equation*}
  D(\alpha)=\supp\iota(\alpha)
  =\{e\in E:e(\alpha)=1\}.
\end{equation*}

Then the formula of \cref{prop:fourier-inversion} becomes
\begin{equation}\label{eq:mu-D}
  \mu(\alpha)=\sum_{e\in D(\alpha)}\lambda_e.
\end{equation}

\begin{definition}
A \emph{cocircuit} of $M$ is a nonempty set $D(\alpha)$ that is inclusion-minimal among the supports of nonzero vectors in $C$.
\end{definition}

Equivalently, cocircuits are complements of hyperplanes (\cite{oxley}, Proposition 2.1.6).  A hyperplane is a maximal proper flat, or a maximal proper subset whose linear span has dimension $g-1$.

Let $B\subseteq E$ be a basis.  Since $B$ is a basis of $V^*$, there is a dual basis
\[
  \{\alpha_b^B:b\in B\}\subseteq V,
  \qquad
  b'(\alpha_b^B)=\delta_{b,b'}\quad(b,b'\in B).
\]
If a basis $B$ is fixed, we may omit the superscript in the notation.

\begin{definition}
The \emph{fundamental cocircuit} of $b\in B$ with respect to $B$ is
\[
  D_b(B)=D(\alpha_b^B).
\]
It is the unique cocircuit whose intersection with $B$ is exactly $\{b\}$ \cite[2.1 Exercise 10]{oxley}
\end{definition}.

We recall the following useful standard fact.

\begin{lemma}\label{lem:every-cocircuit-fundamental}
Every cocircuit is fundamental with respect to some basis.
\end{lemma}

\begin{proof}
Let $D$ be a cocircuit and let $H=E\setminus D$ be the complementary hyperplane.  Choose a basis $B_H$ of $H$ and an element $d\in D$.  Then $B=B_H\cup\{d\}$ is a basis of $M$, and $D$ is the unique cocircuit meeting $B$ only in $d$.  Hence $D=D_d(B)$.
\end{proof}

For $S\subseteq E$, the contraction $M/S$ has rank $\rk(M/S)=g-\rk_M(S)$ \cite[Proposition 3.1.6]{oxley}.
In the present representation this has a useful dual description.  Define
\begin{equation*}
  K_S=\{\alpha\in V:e(\alpha)=0\text{ for all }e\in S\}.
\end{equation*}
Then
\begin{equation*}
  \dim K_S=g-\dim\Span(S)=\rk(M/S).
\end{equation*}

We use the notion of a modular set of cocircuits from \cite{pendavingh2013skew}:

\begin{definition}\label{def:modulartriple}
Three distinct cocircuits $D_1,D_2,D_3$ form a \emph{modular triple} if, with $U=D_1\cup D_2\cup D_3$, and $S=E\setminus U,$ one has $\rk(M/S)=2$. Equivalently, the subspace $K_S\subseteq V$ has dimension $2$.
\end{definition}

We now turn our attention to a construction of modular triples through a single basis exchange. This is done with slightly different notation in \cite{pendavingh2013skew}. Suppose $B$ is a basis, $b\in B$, $f\notin B$, and that $B'=B-b+f$ is again a basis. Denote the dual basis for $B$ by $\alpha_b$.  Then $f(\alpha_b)=1$.  The dual basis for $B'$ is
\begin{equation*}
  \alpha_f'=\alpha_b,
  \qquad
  \alpha_c'=\alpha_c+f(\alpha_c)\alpha_b
  \quad(c\in B\setminus\{b\}).
\end{equation*}
Consequently,
\begin{equation}\label{eq:cocircuit-pivot}
  D_f(B')=D_b(B),
\end{equation}
and, for $c\in B\setminus\{b\}$,
\begin{equation}\label{eq:cocircuit-pivot-2}
  D_c(B')=
  \begin{cases}
    D_c(B),& f\notin D_c(B),\\
    D_c(B)\triangle D_b(B),& f\in D_c(B),
  \end{cases}
\end{equation}
where $\triangle$ denotes symmetric difference.

In the nontrivial latter case, the three cocircuits $D_b(B),D_c(B)$, and $D_c(B')$ form a modular triple. Indeed, their corresponding classes are $\alpha_b,\alpha_c$, and $\alpha_b+\alpha_c$. Their union avoids every basis element in $B\setminus\{b,c\}$, so the complement has rank at least $g-2$. On the other hand, the two independent classes $\alpha_b,\alpha_c$ vanish on that complement, so its rank is at most $g-2$.

\section{Cocircuits and Voronoi-relevant vectors}\label{sec:cocircuits}

It is an elementary observation that the vonorm function is
subadditive, i.e. $$\mu(\alpha + \beta) \leq \mu(\alpha) + \mu(\beta).$$ We see this by choosing minimizers $x\in\alpha$ and $y\in\beta$.  Both $x+y$ and $x-y$ lie in the parity class $\alpha+\beta$.  The parallelogram identity gives
$Q[x+y]+Q[x-y]=2Q[x]+2Q[y]$. At least one of $Q[x+y],Q[x-y]$ is at most $Q[x]+Q[y]$.

With the assumption that conorms are nonnegative, we are able to exactly compute the defect $\mu(\alpha) + \mu(\beta) - \mu(\alpha + \beta)$.

\begin{lemma}\label{lem:defect}
For all $\alpha,\beta\in V$,
\begin{equation*}
  \mu(\alpha)+\mu(\beta)-\mu(\alpha+\beta)
  =2\sum_{e\in D(\alpha)\cap D(\beta)}\lambda_e.
\end{equation*}
In particular, since $\lambda_e>0$ for $e\in E$, equality in subadditivity occurs exactly when $D(\alpha)\cap D(\beta)=\emptyset$.
\end{lemma}

\begin{proof}
For each $e$, the identity
\[
  \mathbf 1_{e(\alpha)=1}+\mathbf 1_{e(\beta)=1}
  -\mathbf 1_{e(\alpha+\beta)=1}
  =2\mathbf 1_{e(\alpha)=e(\beta)=1}
\]
and \eqref{eq:mu-D} give the result.
\end{proof}

Recall the following definition.
\begin{definition}\label{def:vcell}
The \emph{Voronoi cell} of $(L,Q)$ at the origin is
\[
  \Vor(L,Q)=\{x\in L_{\RR}:Q[x]\le Q[x-v]\text{ for every }v\in L\}.
\]
Equivalently,
\[
  \Vor(L,Q)=
  \bigcap_{0\ne v\in L}
  \{x:2\angles{x}{v}\le Q[v]\}.
\]
A nonzero $v\in L$ is a \emph{Voronoi-relevant vector} if
\[
  2\angles{x}{v}=Q[v]
\]
is a facet-defining hyperplane of the Voronoi cell.
\end{definition}
 
 We also recall the following classical lemma from Conway--Sloane.

\begin{lemma}\cite[Theorem 2]{CS}\label{lem:voronoi-parity}
A nonzero vector $v\in L$ is a  Voronoi-relevant vector if and only if $\pm v$ are the only shortest vectors in the parity class $v+2L$.
\end{lemma}

\begin{proof}
The midpoint $v/2$ is equidistant from $0$ and $v$.  For $w\in L$,
\[
  Q\left[\frac v2-w\right]=\frac14Q[v-2w].
\]
Thus $v/2$ lies in $\Vor(L,Q)$ exactly when $v$ is shortest in $v+2L$.  It lies in the relative interior of the facet separating the cells centered at $0$ and $v$ exactly when all other inequalities are strict. Equality with another point $w\notin\{0,v\}$ is equivalent to a vector $v-2w\ne\pm v$ of the same length in $v+2L$.
\end{proof}

For graph flow lattices \cite{amini2010lattice, bacher1997lattice} and, more generally, zonotopal \cite{vallentin2004note} or regular-matroid lattices \cite{dancsolim}, descriptions of the Voronoi-relevant vectors are known. In the next proposition we obtain a description directly from nonnegative conorms, before establishing regularity of the support matroid.

\begin{proposition}\label{prop:cocircuit-shortest}
Let $0\ne\gamma\in V$.  The following are equivalent.
\begin{enumerate}
\item $D(\gamma)$ is a cocircuit of $M$.
\item Every decomposition $\gamma=\alpha+\beta$ with $\alpha,\beta\ne0$ satisfies
\[
  \mu(\gamma)<\mu(\alpha)+\mu(\beta).
\]
\item The shortest vectors in the parity class $\gamma$ are exactly one antipodal pair $\{\pm r_\gamma\}$.
\item $r_\gamma$ is a Voronoi-relevant vector.
\end{enumerate}
\end{proposition}
\begin{proof}
    We will prove equivalence between every consecutive pair of statements. For the first pair, suppose there is a decomposition $\gamma=\alpha+\beta$ with  $\mu(\gamma)=\mu(\alpha)+\mu(\beta)$. Then by \cref{lem:defect} $D(\alpha) \cap D(\beta) = \emptyset$. This implies that $D(\gamma)$ is not support-minimal as $D(\alpha)$ is a proper nonempty subset. Conversely, if $D(\gamma)$ is not support-minimal, it contains some nonzero $\alpha \in V$ with $D(\alpha) \subset D(\gamma)$. Set $\beta=\gamma+\alpha$, and observe that $D(\alpha)$ and $D(\beta)$ are disjoint, nonempty, and union to $D(\gamma)$, so equality holds in subadditivity by \cref{lem:defect}.

    For the second pair, suppose $x \neq \pm y$ are shortest vectors in the class $\gamma$. Then $x-y, x+y \in 2L$. Let $\alpha = \left[ \frac{x-y}{2}\right]$ and $\beta = \left[ \frac{x+y}{2}\right]$. 
   Then $\alpha + \beta = \gamma$ and we note that $\alpha,\beta\neq0$. If $\alpha=0$, then
$(x-y)/2\in2L$, so $(x+y)/2$ is another representative of
$\gamma$. Since $x\neq y$, strict convexity gives
\[
    Q\left[\frac{x+y}{2}\right]
    <\frac{Q[x]+Q[y]}{2}
    =\mu(\gamma),
\]
a contradiction. Similarly, we can conclude $\beta \neq 0$.
Then $$Q\left[ \frac{x-y}{2}\right] + Q\left[ \frac{x+y}{2}\right] = \frac{1}{2}(Q[x] + Q[y]) = \mu(\gamma).$$
    We know $\mu(\gamma) \leq \mu(\alpha) + \mu(\beta)$ but $\mu(\alpha) \leq Q\left[ \frac{x-y}{2}\right]$ and $\mu(\beta) \leq Q\left[ \frac{x+y}{2}\right]$, so $\mu(\gamma) = \mu(\alpha) + \mu(\beta)$.

    Conversely, suppose $\mu(\gamma) = \mu(\alpha) + \mu(\beta)$ where $\alpha,\beta\neq0$, and choose minimizers $a \in \alpha, b \in \beta$. Let $x = a+b, y = a-b$. Then $Q[x] + Q [y] = 2\mu(\alpha) + 2 \mu(\beta) = 2\mu(\gamma)$. Each of $Q[x], Q[y]$ must be at least $\mu(\gamma)$, so each is equal. Then there are two non-antipodal shortest vectors.

    Finally, for the equivalence between the last pair, this is simply \cref{lem:voronoi-parity}.

\end{proof}

We have associated to every cocircuit $D$ a shortest
lattice lift $r_D$, uniquely determined up to sign. Next we show that these lifts collectively see the entire integral lattice, and not just its reduction modulo $2$. For a cocircuit $D$, denote by $\alpha_D$ its unique class in $V$ and choose one of its two shortest representatives:
\begin{equation}\label{eq:rD}
  r_D\in L,
  \qquad [r_D]=\alpha_D,
  \qquad Q[r_D]=\mu(\alpha_D).
\end{equation}
The choice is only up to sign.  All later constructions are sign-independent or explicitly allow signs.

The following lemma is standard but useful. 

\begin{lemma}\label{lem:relevant-generate}
The Voronoi-relevant vectors generate $L$ as an abelian group.  Consequently, the vectors $r_D$, as $D$ ranges over the cocircuits of $M$, generate $L$.
\end{lemma}

\begin{proof}
The translates $\Vor(L,Q)+z$, $z\in L$, form a facet-to-facet tiling of $L_{\RR}$.  Given $z\in L$, choose a path from the interior of $\Vor(L,Q)$ to the interior of $\Vor(L,Q)+z$ that crosses only facets. The centers of adjacent cells differ by a Voronoi-relevant vector.  Summing these differences gives $z$.
\end{proof}

\begin{remark}
    The statement of \cref{lem:relevant-generate} is stronger than saying that the cocircuit classes span $L/2L$.  Spanning modulo $2$ would allow a nontrivial odd-index sublattice.  Eliminating that odd-index possibility is essential in \cref{sec:modulartriples}.
\end{remark}

In the next section we determine the local relations among these lifts and use them to show that the fundamental cocircuit lifts associated to any matroid basis form a $\ZZ$-basis of $L$.

\section{Modular triples and integral cocircuit bases}\label{sec:modulartriples}

The next theorem is the heart of the proof of \cref{thm:main}. It shows that shortest cocircuit lifts satisfy the modular-triple relations which are characteristic of regular chain-group representations.

\begin{theorem}\label{thm:metric-modular-triple}
Let $D_1,D_2,D_3$ be a modular triple of cocircuits.  Let $\alpha_i=\alpha_{D_i}\in V$ and $r_i=r_{D_i}\in L.$ Then there are signs $\eps_i\in\{\pm1\}$ such that
\begin{equation*}
  \eps_1r_1+\eps_2r_2+\eps_3r_3=0.
\end{equation*}
\end{theorem}
\begin{proof}
     Let $U = D_1 \cup D_2 \cup D_3$ and $S = E \setminus U$. Then the rank $\rk(M/S) = 2$, and so $\dim K_S = \dim\{\alpha \in V: e(\alpha) = 0$ for all  $e \in S\} = 2$. Then $K_S$ has four elements, three of which are nonzero. Observe that for each $\alpha_i$, we have $e(\alpha_i) = 0$ for each $e \in S$, so $\alpha_i \in K_S$. Since they are distinct, these must be exactly the three nonzero elements of $K_S$, and thus $\alpha_1 + \alpha_2 + \alpha_3 = 0$. For each $e \in U$, we have $e(\alpha_1) + e(\alpha_2) + e(\alpha_3) = 0$, and at least one of the summands must be nonzero. Therefore, it must be the case that exactly two of the $e(\alpha_i)$ are equal to $1$. So, each $e \in U$ belongs to exactly two of $D_1, D_2, D_3$.

    Now consider 
    \begin{equation*}
\begin{aligned}
  w_0&=\frac{r_1+r_2+r_3}{2},&
  w_1&=\frac{r_1-r_2-r_3}{2},\\
  w_2&=\frac{-r_1+r_2-r_3}{2},&
  w_3&=\frac{-r_1-r_2+r_3}{2}.
\end{aligned}
\end{equation*}
They lie in $L$ since the $\alpha_i$ sum to $0$: we have $[r_1] + [r_2]+[r_3] = 0$, so $r_1+r_2+r_3 \in 2L$, and similar for all other sign patterns. If $t=[w_0]\in V$, then their four parity classes are
\[
  t+K_S=\{t,t+\alpha_1,t+\alpha_2,t+\alpha_3\}.
\]

The four sign vectors are the rows of a Hadamard matrix with orthogonal columns. This gives 
\begin{equation*}
  \sum_{i=0}^3Q[w_i]
  =Q[r_1]+Q[r_2]+Q[r_3].
\end{equation*}
Since each $r_i$ is a shortest vector in its class and every $e\in U$ occurs in exactly two $D_i$, the sum is
\begin{equation*}
  Q[r_1]+Q[r_2]+Q[r_3]
  =\mu(\alpha_1)+\mu(\alpha_2)+\mu(\alpha_3)
  =2\sum_{e\in U}\lambda_e.
\end{equation*}

Now consider $\sum_{\alpha \in t+K_S}\mu(\alpha)$, which we can compute using \eqref{eq:mu-D}. If $e \in U$ then $e|_{K_S} \neq 0$, so $e = 1$ at exactly two points in $t+K_S$. If $e \in E\setminus U = S$ then $e|_{K_S} = 0$, so $e$ is constant on $t + K_S$. Thus,
$$\sum_{\alpha \in t+K_S}\mu(\alpha) = 2\sum_{e \in U} \la_e + 4 \sum_{\substack{e\in S\\e(t)=1}}\lambda_e.$$

Furthermore, $\sum_{i=0}^3Q[w_i] \geq \sum_{\alpha \in t+K_S} \mu(\alpha)$, so $$2\sum_{e \in U} \la_e \geq 2\sum_{e \in U} \la_e + 4 \sum_{\substack{e\in S\\e(t)=1}}\lambda_e.$$

Since all $\la_e$ for $e \in E$ are strictly positive, we conclude that $e(t) = 0$ for all $e \in S$, so $t+K_S = K_S$. This means one of $w_i$ is in the class $[0]$. Equality holds throughout the preceding inequalities, so each $w_i$ necessarily attains the minimum in its parity class. One of those classes is zero, so one of the $w_i$ must be equal to $0$. Therefore, the sign relation associated to that $w_i$ holds.
\end{proof}

\begin{remark}
For a cographic matroid, a modular triple of cocircuits is the matroidal shadow of a theta configuration in a graph.  The signed relation of \cref{thm:metric-modular-triple} is then the familiar signed relation among its three cycle flows.  The theorem says that conorm positivity forces arbitrary shortest cocircuit lifts to obey the same local flow law.
\end{remark}

Recent work of M\'enab\'e \cite{menabe} studies signed three-term dependencies among Voronoi-relevant vectors under the name ``line relations," and states that these relations generate the full lattice of dependencies among the relevant vectors. In the nonnegative-conorm setting considered here, \cref{thm:metric-modular-triple} shows that every modular triple of cocircuits gives such a line relation. Modular triples form a distinguished subclass of line relations.

In the remainder of this section, we will show that fundamental cocircuit lifts form a $\ZZ$-basis of $L$. For a basis $B$ of $M$, define
\begin{equation*}
  K_B=\left\langle r_{D_b(B)}:b\in B\right\rangle_{\ZZ}\subseteq L.
\end{equation*}
The classes $\alpha_b$ form a basis of $V=L/2L$, so the vectors $r_{D_b(B)}$ are linearly independent and $K_B$ has full rank.  At this point one knows only that $[L:K_B]$ is odd.

We now propagate the modular-triple relations through the basis graph. The following basis-exchange argument is an integral analogue of \cite[Lemma 3.21]{pendavingh2013skew}. 

\begin{proposition}\label{prop:KB-invariant}
If $B$ and $B'$ are bases of $M$, then $K_B=K_{B'}$.
\end{proposition}

\begin{proof}
The basis graph of a matroid is connected, so it is enough to consider a single exchange
\[
  B'=B-b+f.
\]
By \eqref{eq:cocircuit-pivot}, the new fundamental cocircuit at $f$ equals the old one at $b$.  For $c\in B\setminus\{b\}$, either the fundamental cocircuit is unchanged or, by \eqref{eq:cocircuit-pivot-2}, the old and new cocircuits belong to a modular triple.  In the latter case, \cref{thm:metric-modular-triple} gives
\[
  r_{D_c(B')}=\pm r_{D_c(B)}\pm r_{D_b(B)}.
\]
Thus the new generating set is obtained from the old one by signed elementary integral operations, so $K_{B'}=K_B$.
\end{proof}

\begin{theorem}\label{thm:integral-cocircuit-basis}
For every basis $B$ of $M$, the vectors
\[
  \{r_{D_b(B)}:b\in B\}
\]
form a $\ZZ$-basis of $L$.
\end{theorem}

\begin{proof}
By \cref{prop:KB-invariant}, the sublattice $K_B$ is independent of $B$.  Every cocircuit is fundamental for some basis by \cref{lem:every-cocircuit-fundamental}, so every shortest cocircuit lift $r_D$ lies in this common sublattice.  \cref{lem:relevant-generate} tells us that those lifts generate all of $L$.  Therefore $K_B=L$.
\end{proof}

Theorem~\ref{thm:integral-cocircuit-basis} removes the possible odd-index discrepancy between the mod-$2$ dual basis and the integral lattice. We can now reconstruct the coordinate functionals that, in the graph case, record the flow through individual edges.

\section{Integral coedges and regularity}\label{sec:regularity}

Fix $e\in E$.  Analogous to the graph construction, where covectors record whether a graph cycle uses a particular edge, we now look for primitive covectors $a_e\in L^*$ that record whether a cocircuit contains $e$.

Choose a basis $B$ which contains $e$ and define
\begin{equation*}
  L_e(B)=
  \left\langle
  r_{D_b(B)}:b\in B\setminus\{e\}
  \right\rangle_{\ZZ}.
\end{equation*}
By \cref{thm:integral-cocircuit-basis}, this is a primitive rank-$(g-1)$ sublattice of $L$.

\begin{lemma}\label{lem:Le-independent}
The lattice $L_e(B)$ is independent of the choice of basis $B$ containing $e$.
\end{lemma}

\begin{proof}
The bases containing $e$ correspond, after deleting $e$, to the bases of the contraction $M/e$. The basis graph of the contraction is connected, so the bases containing $e$ in $M$ are connected via basis exchange through exchange operations avoiding $e$.  Consider an exchange $B'=B-b+f$
with $b\ne e$, so that both bases contain $e$. For $c\in B\setminus{b,e}$, \eqref{eq:cocircuit-pivot-2} shows that the corresponding new lift is either unchanged or differs from the old lift by $\pm r_{D_B(b)}$, while the lift indexed by $f$ equals $\pm r_{D_b(B)}.$ Hence the sets of generators indexed away from $e$ are related by signed elementary integral operations, and therefore $L_e(B')=L_e(B)$.
\end{proof}

According to \cref{lem:Le-independent}, we can write the lattice simply as $L_e$.

\begin{lemma}\label{lem:Le-cocircuit-description}
We have
\begin{equation*}
  L_e=\left\langle r_D:D\text{ a cocircuit and }e\notin D\right\rangle_{\ZZ}.
\end{equation*}
Furthermore, if $D$ is a cocircuit containing $e$, then the image of $r_D$ generates the quotient $L/L_e\cong\ZZ$.
\end{lemma}

\begin{proof}
If $b\in B\setminus\{e\}$, then $D_b(B)\cap B=\{b\}$, so $e\notin D_b(B)$. So $L_e$ is generated by cocircuit lifts avoiding $e$.

Conversely, let $D$ be a cocircuit avoiding $e$.  Its complement $H=E\setminus D$ is a hyperplane containing $e$.  Choose a basis $B_H$ of $H$ that contains $e$, and choose $d\in D$.  Then $B=B_H\cup\{d\}$ is a basis containing $e$, and $D=D_d(B)$.  Hence $r_D\in L_e$.

Now let $D$ be a cocircuit that contains $e$.  Choose a basis $B_H$ of the hyperplane $E\setminus D$ and set $B=B_H\cup\{e\}$.  Then $D=D_e(B)$.  Then $r_D$, together with the $g-1$ generators of $L_e(B)$, is a $\ZZ$-basis of $L$.  Thus its class generates $L/L_e$.
\end{proof}

Because $L_e$ is primitive of corank one, $L/L_e\cong\ZZ$.  Choose a primitive covector $a_e\in L^*$, unique up to sign, with  $\ker a_e=L_e$. We may normalize it by requiring $a_e(r_D)=1$ for one chosen cocircuit $D$ containing $e$. We call $a_e$ the \emph{reconstructed coedge}, as an analogy with the graph case.

\begin{proposition}\label{prop:coedge-rule}
For every cocircuit $D$,
\begin{equation*}
  a_e(r_D)=
  \begin{cases}
    0,&e\notin D,\\
    \pm1,&e\in D.
  \end{cases}
\end{equation*}
Furthermore, the reduction of $a_e$ modulo $2$ is the original character $e\in V^*$.
\end{proposition}

\begin{proof}
If $e\notin D$, then $r_D\in L_e$ by \cref{lem:Le-cocircuit-description}, so $a_e(r_D)=0$.  If $e\in D$, then $r_D$ generates $L/L_e$, so a primitive homomorphism $L/L_e\to\ZZ$ takes it to $\pm1$.

As for the reduction modulo $2$, pick a basis $B$ containing $e$.  The classes of the fundamental lifts indexed by $B\setminus\{e\}$ span $\ker(e)\subseteq V$, and their integral span is $L_e$. So $a_e\bmod2$ vanishes on $\ker(e)$.  It is nonzero on the remaining one-dimensional quotient, since it takes the odd value $\pm1$ on $r_{D_e(B)}$. Therefore $a_e\bmod 2=e$.
\end{proof}

Now we use the coedges to construct a unimodular matrix representing the binary support matroid. Choose a $\ZZ$-basis of $L$.  Regard each $a_e$ as a column in $\ZZ^g$, and let
\[
  A=(a_e)_{e\in E}\in\ZZ^{g\times |E|}.
\]

\begin{definition}
A rank-$g$ integral matrix is \emph{unimodular} if every $g\times g$ minor is $0$ or $\pm1$.  It is \emph{totally unimodular} if every square minor is $0$ or $\pm1$.
\end{definition}

\begin{remark}
    Unimodularity is also called ``weak unimodularity" in parts of the matroid literature.
\end{remark}

Unimodularity is invariant under changing the $\ZZ$-basis of $L$. Total unimodularity is not: multiplying on the left by a matrix in $\GL_g(\ZZ)$ can create entries larger than $1$.  However, a full-rank unimodular matrix will become totally unimodular after normalizing one basis block to the identity.

\cref{prop:coedge-rule} shows that the reconstructed covectors behave exactly like the coordinates of signed cocircuit vectors in a regular representation. We now verify that they provide a unimodular integral lift of the original binary support matroid.

\begin{theorem}\label{thm:A-unimodular}
The column matroid of $A$ is $M$, and $A$ is unimodular. After an integral change of basis of $L$, $A$ is totally unimodular, therefore $M$ is a regular matroid.
\end{theorem}
\begin{proof}
    Let $B$ be a basis of $M$.  Use the $\ZZ$-basis $\{r_{D_b(B)}:b\in B\}$ of $L$.  For $b,e\in B$, the fundamental cocircuit $D_b(B)$ contains $e$ exactly when $e=b$.  By \cref{prop:coedge-rule}, the pairing matrix is diagonal with diagonal entries $\pm1$, and therefore the determinant $\det A_B=\pm1$. 

    Now let $S \subset E$ be a subset of size $g$ which is not a basis. It must lie in some hyperplane $H$. The complement $D=E\setminus H$ is a cocircuit disjoint from $S$.  For every $e\in S$, \cref{prop:coedge-rule} gives $a_e(r_D)=0$, and therefore the corresponding determinant $\det A_S = 0$. Thus the nonzero minors are exactly the basis minors and are all equal to $\pm 1$, as desired.

    For total unimodularity it is a standard argument: choose one basis $B$ and multiply on the left by $A_B^{-1}\in\GL_g(\ZZ)$.  The resulting matrix has the form
\[
  [I_g\mid N].
\]
Every square minor of $N$ can be completed to a maximal minor by adjoining complementary identity columns, so it is $0$ or $\pm1$.  Expanding any mixed minor of $[I_g\mid N]$ along its identity columns reduces it to a minor of $N$ or to zero.  Hence $[I_g\mid N]$ is totally unimodular, and the matroid $M$ is regular.
\end{proof}

\section{Reconstructing the quadratic form}\label{sec:reconstructing}

Using the integral covectors, define a quadratic form
\begin{equation*}
  Q_A[x]=\sum_{e\in E}\lambda_e a_e(x)^2.
\end{equation*}

In the previous section, we showed that $Q_A$ is matroidal. It remains to prove that $Q_A=Q$. We note that it is not enough to say that $Q$ and $Q_A$ have the same parity-minimum table. The table is invariant under integral basis changes acting trivially modulo 2. Injectivity of the full vonorm table is a separate and subtle problem.  Instead, we construct one common representative in every parity class on which the two forms agree. Then we can apply a parity-rigidity lemma of \cite{sikiric2022iso}.

We recall the following standard property of binary cocycles (see \cite[Theorem 9.1]{oxley}, by duality.
\begin{lemma}\label{lem:disjoint-cocircuits}
Every nonzero cocycle $c\in C$ is a disjoint union of cocircuits:
\[
  \supp(c)=D_1\sqcup\cdots\sqcup D_k.
\]
\end{lemma}

\begin{proof}
Choose a nonzero cocycle of inclusion-minimal support among the nonzero cocycles whose support is contained in $\supp(c)$. Its support is a cocircuit, call it $D_1$.  Replace $c$ by $c+D_1$.  Because $D_1\subseteq\supp(c)$, this removes the coordinates of $D_1$.  We repeat this until the support is empty.
\end{proof}

\begin{lemma}\label{lem:orthogonal}
If $D_1,\ldots,D_k$ are pairwise disjoint cocircuits, then
\[
  \angles{r_{D_i}}{r_{D_j}}=0 \quad \text{ for all } i \neq j \in [k].
\]

Further, for every choice of signs
$\varepsilon_1,\ldots,\varepsilon_k\in\{\pm1\}$, the vector
\[
    x=\sum_{i=1}^k \varepsilon_i r_{D_i}
\]
is a shortest vector in the class $\alpha_{D_1}+\cdots+\alpha_{D_k}$. In particular,
\[
    Q[x]
    =
    \sum_{i=1}^k Q[r_{D_i}]
    =
    \mu\left(\alpha_{D_1}+\cdots+\alpha_{D_k}\right).
\]
\end{lemma}

\begin{proof}
We begin with pairwise orthogonality. By \cref{lem:defect}, disjointness of $D_i, D_j$ gives
\[
  \mu(\alpha_{D_i}+\alpha_{D_j})
  =\mu(\alpha_{D_i})+\mu(\alpha_{D_j}).
\]
Both $r_{D_i}+r_{D_j}$ and $r_{D_i}-r_{D_j}$ lie in the class $\alpha_{D_i}+\alpha_{D_j}$.  We have
\[
  Q[r_{D_i} - r_{D_j}] + Q[r_{D_i} + r_{D_j}] = 2Q[r_{D_i}]+2Q[r_{D_j}]
  =2\mu(\alpha_{D_i}+\alpha_{D_j}).
\]
Each of $Q[r_{D_i} - r_{D_j}], Q[r_{D_i} + r_{D_j}]$ is at least the minimum $\mu(\alpha_{D_i}+\alpha_{D_j})$, so both attain it.  On the other hand, subtracting them gives
\[
  Q[r_{D_i} + r_{D_j}] - Q[r_{D_i} - r_{D_j}]  = 4\angles{r_{D_i}}{r_{D_j}} = 0.
\]

Now let $x=\sum_{i=1}^k \varepsilon_i r_{D_i}.$ Since signs do not affect reduction modulo $2L$, we have $[x] = \alpha_{D_1}+\cdots+\alpha_{D_k}.$ By pairwise orthogonality,
\[
    Q[x]
    =
    \sum_{i=1}^k Q[r_{D_i}]
    =
    \sum_{i=1}^k \mu(\alpha_{D_i}).
\]
Because the $D_i$ are pairwise disjoint,
\[
    D\left(\alpha_{D_1}+\cdots+\alpha_{D_k}\right)
    =
    D_1\sqcup\cdots\sqcup D_k.
\]
Then the formula \eqref{eq:mu-D} gives
\[
    \mu\left(\alpha_{D_1}+\cdots+\alpha_{D_k}\right) =
    \sum_{e\in D_1\sqcup\cdots\sqcup D_k}\lambda_e =
    \sum_{i=1}^k\sum_{e\in D_i}\lambda_e =
    \sum_{i=1}^k\mu(\alpha_{D_i}).
\]
And then
\[
    Q[x]
    =
    \mu\left(\alpha_{D_1}+\cdots+\alpha_{D_k}\right),
\]
so $x$ is a shortest vector in its class.

\end{proof}

Let $\alpha\in V$ and decompose its cocycle as
\[
  D(\alpha)=D_1\sqcup\cdots\sqcup D_k
\]
according to \cref{lem:disjoint-cocircuits}.  Define
\begin{equation*}
  x_\alpha=r_{D_1}+\cdots+r_{D_k}.
\end{equation*}
For $\alpha=0$, we take $x_0=0$.

\begin{proposition}\label{prop:common-representative}
For every $\alpha\in V$, the vector $x_\alpha$ lies in the class $\alpha$ and simultaneously minimizes $Q$ and
$Q_A$. That is, we have
\[
    Q[x_\alpha]
    =
    \min_{\substack{x\in L\\ [x]=\alpha}}Q[x]
    =
    \mu(\alpha)
    =
    \min_{\substack{x\in L\\ [x]=\alpha}}Q_A[x]
    =
    Q_A[x_\alpha].
\]
\end{proposition}
\begin{proof}
    The cocycle with support $D_i$ is $\iota(\alpha_{D_i})$. Since the
supports $D_i$ are disjoint and their union is $D(\alpha)$, we have
\[
    \iota(\alpha)
    =
    \iota(\alpha_{D_1})
    +\cdots+
    \iota(\alpha_{D_k})
\]
in $\mathbb F_2^E$. Because $\iota$ is injective, we have $\alpha = \alpha_{D_1}+\cdots+\alpha_{D_k}.$ Consequently, $[x_\alpha]=\alpha$.

By \cref{lem:orthogonal},
\[
    Q[x_\alpha] =
    \sum_{i=1}^k Q[r_{D_i}] =
    \sum_{i=1}^k \mu(\alpha_{D_i}) =
    \sum_{i=1}^k\sum_{e\in D_i}\lambda_e =
    \sum_{e\in D(\alpha)}\lambda_e =
    \mu(\alpha).
\]
Thus $x_\alpha$ minimizes $Q$ in its parity class.

Next, we evaluate $Q_A$ on $x_\alpha$. Fix $e\in E$. Since the
cocircuits $D_1,\ldots,D_k$ are pairwise disjoint, the element $e$
belongs to at most one of them. Then, from \cref{prop:coedge-rule},
it follows that
\[
    a_e(x_\alpha)^2
    =
    \begin{cases}
        1, & e\in D(\alpha),\\
        0, & e\notin D(\alpha).
    \end{cases}
\]
And then
\[
    Q_A[x_\alpha]
    =
    \sum_{e\in E}\lambda_e a_e(x_\alpha)^2
    =
    \sum_{e\in D(\alpha)}\lambda_e
    =
    \mu(\alpha).
\]

It remains to verify that $x_\alpha$ is a shortest vector for $Q_A$ in its parity class. Let $y\in L$ be any representative of $\alpha$. Since the reduction of $a_e$ modulo $2$ is $e$, we have $a_e(y)\equiv e(\alpha)\pmod 2.$ In particular, for every $e\in D(\alpha)$, the integer $a_e(y)$ is
odd, and hence $a_e(y)^2\geq1.$ Since all $\lambda_e$ are positive, we have
\begin{align*}
    Q_A[y]
    &=
    \sum_{e\in E}\lambda_e a_e(y)^2 \\
    &\geq
    \sum_{e\in D(\alpha)}\lambda_e =
    \mu(\alpha) =
    Q_A[x_\alpha].
\end{align*}
and we are done.
\end{proof}

\begin{lemma}\cite[Lemma 3.1]{sikiric2022iso}\label{lem:parity-rigidity}
Let $H$ be a real quadratic form on $L_{\RR}$.  Suppose that for every class $\alpha\in L/2L$ there is a representative $x_\alpha\in\alpha$ such that
\[
  H[x_\alpha]=0.
\]
Then $H=0$.
\end{lemma}

\begin{proof}
Choose a basis $L\cong\ZZ^g$.  The conditions $H[x_\alpha]=0$ are homogeneous linear equations with integer coefficients in the coefficients of $H$.  If they had a nonzero real solution, they would have a nonzero rational solution.  After scaling, choose a nonzero solution whose polynomial coefficients are integers and not all even.

Reduce the resulting homogeneous quadratic polynomial modulo $2$.  Because $x_\alpha\equiv\alpha\pmod2$, the reduced polynomial vanishes at every point of $\Ftwo^g$.  Evaluating at the standard basis vector $e_i$ shows that every coefficient of $x_i^2$ is zero.  Evaluating at $e_i+e_j$ then shows that every coefficient of $x_ix_j$ is zero.  This contradicts the choice that not all coefficients were even.
\end{proof}

\begin{theorem}\label{thm:Q-reconstruction}
    In any integral basis of $L$,
\[
  Q=A\diag(\lambda_e)A^{\mathsf T}.
\]
That is,
\[Q[x]=\sum_{e\in E}\lambda_e a_e(x)^2
  \qquad(x\in L_{\RR}).\]
\end{theorem}
\begin{proof}
    Let $H=Q-Q_A$.  By \cref{prop:common-representative}, for every parity class $\alpha$ there is a representative $x_\alpha$ with $H[x_\alpha]=0$.  Finally, apply \cref{lem:parity-rigidity}.
\end{proof}

We have now reconstructed both the integral representation and the metric coefficients of $Q$. The main theorem and its tropical Schottky consequences follow immediately, in the upcoming section.

The reconstruction theorem also yields the following rigidity statement for the vonorm table on the nonnegative-conorm locus.

\begin{corollary}
Let $Q$ and $Q'$ be positive-definite quadratic forms with
nonnegative conorms. If they have the same vonorm table, then they are integrally equivalent.
\end{corollary}

\begin{proof}
By \Cref{thm:A-unimodular,thm:Q-reconstruction}, both $Q$ and $Q'$ lie in
the matroidal locus. The result now follows from
\cite[Corollary~6.3]{sikiric2022iso}.
\end{proof}

\section{Main theorem, the Tropical Schottky problem, and consequences}\label{sec:main}

\begin{theorem}\label{thm:main}
Let $Q$ be positive-definite.  If all conorms $\lambda_e$ defined in \eqref{eq:lambda-def} are nonnegative, then $Q$ is matroidal.  That is, the construction above produces primitive integral covectors $a_e\in L^*$ for the positive conorm support $E$ such that
\begin{enumerate}
\item $a_e\bmod2=e$;
\item $(a_e)_{e\in E}$ is a unimodular configuration representing the binary support matroid $M(E)$;
\item $Q=\sum_{e\in E}\lambda_ea_e^2$.
\end{enumerate}
Consequently, the support matroid $M(E)$ is regular.
\end{theorem}

\begin{proof}
Combine \cref{prop:coedge-rule,thm:A-unimodular,thm:Q-reconstruction}.
\end{proof}

\cref{thm:main} proves \cref{conj:DSK} in every dimension:

\begin{corollary}
\label{cor:nonnegative-conorm-criterion}
A positive-definite quadratic form is matroidal if and only if all of its conorms are nonnegative.
\end{corollary}

\begin{proof}
The implication from nonnegative conorms to matroidality follows from Theorem~\ref{thm:main}. Conversely, after columns with the same nonzero mod-2 class are aggregated, the conorm calculation of \cite[Proposition 6.2] {sikiric2022iso} identifies the nonzero conorms with the resulting positive coefficients, while remaining conorms vanish.
\end{proof}

Recall that a matroid is \emph{cographic} if it is isomorphic to the dual $M^*(G)$ of the graphic matroid of a graph $G$. In a cographic matroid, cocircuits correspond to simple cycles of $G$. The fact that the image of the tropical Torelli map is the cographic subfan is classical \cite[Theorem~5.2.4]{brannetti2011torelli}. We next extend the result of \cite{sikiric2022iso} for $g \leq 5$, and affirmatively answer \cref{ques:CKS} of \cite{chua2019schottky} in every genus. This gives a criterion to check whether a given quadratic form is in the tropical Schottky locus.

\begin{corollary}\label{cor:schottky}
Let $Q$ be a positive-definite rank-$g$ form.  Then $Q$ is a tropical Riemann matrix of a metric graph if and only if all conorms are nonnegative and the positive conorm support matroid $M(E)$ is cographic.

When these conditions hold, the positive conorms are the edge lengths of the recovered simple cographic representation, after bridges are deleted and coedges with the same mod-2 class are aggregated, and
\[
  Q=A\diag(\lambda_e)A^{\mathsf T}
\]
is an exact graph-period certificate.
\end{corollary}

\begin{proof}
The forward implication is Theorem 4.2 of \cite{chua2019schottky}.  Conversely, \cref{thm:main} produces a unimodular representation $A$ of the cographic support matroid.  A regular matroid has an essentially unique representation \cite[Corollary~10.1.4]{oxley}, up to an integral row change, column signs, and column permutation.  Therefore $A$ is equivalent to a coedge matrix of a graph.  Column signs disappear from the rank-one forms $a_ea_e^{\mathsf T}$, so the expression of $Q$ in \cref{thm:Q-reconstruction} is the period form of that graph with edge lengths $\lambda_e$.
\end{proof}

\begin{remark}
    Bridges are invisible to the flow lattice.  In general, coedges with the same binary class are aggregated by the Fourier transform, so one recovers only the corresponding total length. The canonical output is the cyclic equivalence class of the induced metric $3$-edge-connectivization, in accordance with the tropical Torelli theorem \cite{caporaso2010torelli,brannetti2011torelli}.
\end{remark}

\subsection{Decision and recovery algorithms}

We now record the resulting exact decision-and-recovery procedure.
For convenience, one can suppose that $L=\ZZ^g$. The procedure applies to exact real input data for which the required lattice minimization problems can be solved exactly. In particular, it applies to rational positive-definite forms.

\begin{alg}[Tropical Schottky Decision]
\label{alg:trop-schottky-decision}
\emph{Input:} An exact positive-definite quadratic form $Q$ on a
rank-$g$ lattice $L$.
\emph{Output:} ``Yes'' if $Q\in\mathfrak{J}^{\mathrm{trop}}_g$,
and ``No'' otherwise.

\begin{enumerate}
    \item
    For every parity class $\alpha\in L/2L$, compute
    \[
        \mu(\alpha)
        =
        \min_{\substack{x\in L\\ [x]=\alpha}} Q[x].
    \]
    For use in \Cref{alg:trop-schottky-recovery}, retain one
    shortest representative $r_\alpha\in\alpha$.

    \item
    Compute the conorms
    \[
        \lambda_e
        =
        -2^{1-g}
        \sum_{\alpha\in L/2L}
        (-1)^{e(\alpha)}\mu(\alpha),
        \qquad
        0\neq e\in(L/2L)^*,
    \]
    using the Walsh--Hadamard transform.

    \item
    If $\lambda_e<0$ for some $e$, output ``No'' and stop.
    Otherwise, set
    \[
        E=\{e\in(L/2L)^*:\lambda_e>0\}.
    \]

    \item Form the binary matroid $M(E)$ represented by the elements of $E$. Test whether $M(E)$ is cographic (e.g. by a classical algorithm such as Tutte’s graphic-matroid algorithm \cite{tutte}). If $M(E)$ is not cographic, output ``No.'' Otherwise, recover a graph $G$, together with a labeled matroid isomorphism
\[
    \phi:M^*(G)\xrightarrow{\sim}M(E),
\]
retain the pair $(G,\phi)$ for \cref{alg:trop-schottky-recovery}, and output ``Yes.''
\end{enumerate}
\end{alg}

\begin{alg}[Tropical Schottky Recovery]
\label{alg:trop-schottky-recovery}
\emph{Input:} A positive-definite quadratic form $Q$ for which
\Cref{alg:trop-schottky-decision} returns ``Yes.''
\emph{Output:} A metric graph $\Gamma$, together with an integral
cycle basis whose tropical Riemann matrix is $Q$.

\begin{enumerate}

    \item Let $E$, $(\lambda_e)_{e\in E}$, $(r_\alpha)_{\alpha\in L/2L}$, and $(G,\phi)$ be the data retained by \cref{alg:trop-schottky-decision}. Use $\phi$ to label the edges of $G$ by the elements of $E$.

    \item
    Assign length $\ell_e=\lambda_e$ to the edge labeled by $e$, and set $D=\diag(\lambda_e)_{e\in E}.$

    \item
    For each $e\in E$, choose a matroid basis $B_e$ containing $e$.
    For $b\in B_e$, let $\alpha_b^{B_e}\in L/2L$ denote the vector dual to $b$ with respect to $B_e$, and set
    \[
        L_e
        =
        \left\langle
            r_{\alpha_b^{B_e}}
            :
            b\in B_e\setminus\{e\}
        \right\rangle_{\ZZ}.
    \]
    Compute a primitive covector $a_e\in L^*$ satisfying $\ker(a_e)=L_e,$
    and let $A=(a_e)_{e\in E}.$

    \item
    Compute a coedge matrix $C=(c_e)_{e\in E}$ of $G$, with its columns labeled by $E$. By unique
    representability, compute
    \[
        U\in\GL_g(\ZZ)
        \qquad\text{and}\qquad
        S=\diag(\varepsilon_e)_{e\in E},
        \quad \varepsilon_e\in\{\pm1\},
    \]
    such that $A=UCS.$
    \item
    Reorient the edges according to $S$ and set $C'=CS$. Output
    the metric graph $\Gamma=(G,\ell)$ with the integral cycle basis whose coedge matrix is $UC'$.
    Its Riemann matrix is
    \[
        U C'D(C')^{\mathsf T}U^{\mathsf T}
        =
        A D A^{\mathsf T}
        =
        Q.
    \]
\end{enumerate}
\end{alg}

\begin{theorem}
\label{thm:schottky-algorithms-correct}
\Cref{alg:trop-schottky-decision} outputs ``Yes'' if and only if
$Q\in\mathfrak{J}^{\mathrm{trop}}_g$. In that case,
\Cref{alg:trop-schottky-recovery} outputs a metric graph with an
integral cycle basis whose Riemann matrix is exactly $Q$.
\end{theorem}

\begin{proof}
By \Cref{cor:schottky}, the form $Q$ lies in
$\mathfrak{J}^{\mathrm{trop}}_g$ if and only if all of its 
conorms are nonnegative and its positive conorm support matroid is
cographic. This proves the correctness of
\Cref{alg:trop-schottky-decision}.

Suppose that the decision algorithm returns ``Yes.'' By
\Cref{thm:main}, the covectors constructed in Step~3 of \Cref{alg:trop-schottky-recovery} form a unimodular representation of $M(E)$ and satisfy
$Q=A D A^{\mathsf T}.$ Since $C$ is a coedge representation of the same regular matroid, unique representability gives $A=UCS$ for some $U\in\GL_g(\ZZ)$ and signed diagonal matrix $S$. Since
$SDS^{\mathsf T}=D$, we obtain
\[
    Q
    =
    U(CS)D(CS)^{\mathsf T}U^{\mathsf T}.
\]
Thus the graph and cycle basis returned by the recovery algorithm
have Riemann matrix $Q$.
\end{proof}

\begin{remark}
There are $2^g$ parity classes. Once their minima are known, the Walsh--Hadamard transform requires $O(g2^g)$ arithmetic operations. The principal computational cost is the exact computation of the $2^g$ parity minima, each of which is a closest-vector problem in a prescribed coset of $2L$. 
\end{remark}

\subsection{The remaining Schottky--Igusa presentations}

We end by returning to the genus-four Schottky--Igusa form that motivates the tropical Igusa locus. Chua, Kummer, and Sturmfels show that the presentations with $\dim N'=3$ yield the conorm inequalities, and ask whether the remaining presentations, with $\dim N'\leq2$, impose additional conditions that cut the tropical Igusa locus down to the tropical Schottky locus \cite[Question~4.10]{chua2019schottky}. We show that they do not.

\begin{proposition}\label{prop:remaining-schottky-igusa}
For every choice with $\dim N'\leq2$ in \cite[Question~4.10]{chua2019schottky}, the associated tropical Schottky--Igusa condition is implied by nonnegativity of the conorms. In particular, these presentations do not further cut the tropical Igusa locus.
\end{proposition}

\begin{proof}
We use the notation of \cite[\S4]{chua2019schottky}. Let $W=N'\subseteq\mathbf F_2^4$ and $r=\dim W$. We write a characteristic as $m=(m',m'')$ and let $q(m)=m'\cdot m''$ be its parity quadratic form.

The four compatible $N$-cosets form an affine plane. The three cosets $m_i+N$ are even, while the fourth must be odd by azygeticity. We claim that exactly two of the three even cosets meet $\{0\}\times\mathbf F_2^4$. Indeed, an even compatible coset meeting $\{0\}\times\mathbf F_2^4$ can be written as $(0,y)+N$, where $y\in\mathbf F_2^4$ satisfies
\[
    y\cdot n'=q(n)\qquad\text{for every }n=(n',n'')\in N.
\]
Since $m_1+N$ is even, the polarization identity shows that $q|_N$ is linear. Furthermore, it vanishes on the kernel of the projection $N\to W$, because $q(0,k)=0$. Thus the equations specify a linear functional on $W$ and impose $r$ independent conditions on $y$. They therefore have $2^{4-r}$ solutions.

The projection $N\to W$ has kernel of dimension $3-r$. Then, each $N$-coset meeting $\{0\}\times\mathbf F_2^4$ contains $2^{3-r}$ points of the form $(0,y)$. It follows that exactly $\frac{2^{4-r}}{2^{3-r}}=2$ of the even compatible cosets meet $\{0\}\times\mathbf F_2^4$. A coset $m_i+N$ has this property exactly when its first projection is $W$. After relabeling, we conclude that there is some $b\notin W$ such that
\[
    (m_1+N)'=(m_2+N)'=W,\qquad (m_3+N)'=b+W.
\]
Furthermore, every point of a projected coset has $2^{3-r}$ preimages.

Writing $\pi_i^{\mathrm{trop}}$ for the tropical theta products of \cite[(23)]{chua2019schottky}, Fourier inversion gives
\[
\begin{aligned}
    \pi_1^{\mathrm{trop}}-\pi_3^{\mathrm{trop}}
      &=2^{3-r}\left(\sum_{u\in W}\Theta_u(Q)-\sum_{u\in b+W}\Theta_u(Q)\right)\\
      &=\sum_{\substack{v\in W^\perp\\v(b)=1}}\vartheta_v(Q)
       =2\sum_{\substack{v\in W^\perp\\v(b)=1}}\lambda_v(Q),
\end{aligned}
\]
where the last equality uses $\lambda_v=\frac12\vartheta_v$ in genus four.

Since $(m_1+N)'=(m_2+N)'$, we have $\pi_1^{\mathrm{trop}}=\pi_2^{\mathrm{trop}}$. Therefore the maximum in \cite[(23)]{chua2019schottky} is attained at least twice exactly when $\pi_1^{\mathrm{trop}}\geq\pi_3^{\mathrm{trop}}$, or equivalently when
\[
    \sum_{\substack{v\in W^\perp\\v(b)=1}}\lambda_v(Q)\geq0.
\]
This is automatically satisfied when all conorms are nonnegative.
\end{proof}

 This gives a negative answer to the specific
proposal in \cite[Question~4.10]{chua2019schottky}. Although the Schottky--Igusa modular form cuts out the closure of the classical Schottky locus in genus four, the tropical conditions obtained from the presentations above cut out only the tropical Igusa locus, which by \cref{cor:nonnegative-conorm-criterion} is exactly the matroidal locus. Combined with \cref{cor:schottky}, this shows that these tropicalizations detect matroidality but not the cographicity of the positive conorm support matroid. Nevertheless, the fundamental theorem of tropical geometry indicates that the missing cographicity conditions should be imposed by tropicalizations of further polynomials in the full ideal of the classical Schottky locus in theta-constant coordinates \cite{maclagan2015introduction}. In particular, one should look among polynomial consequences of the Schottky--Igusa equation together with the equations defining the image of the theta-constant map in $\PP^{15}$. Identifying explicit equations of this kind, and understanding why their tropicalizations detect cographicity, remains an interesting problem.

{\small
\bibliography{references}}
\bibliographystyle{abbrv}
\end{document}